\documentclass[12pt]{amsart}
 \usepackage{txfonts}
\usepackage{bbm}
\usepackage{mathrsfs}
\usepackage{amssymb}
\usepackage{amssymb,amsmath,amsthm,amscd}
\usepackage[all]{xy}

\usepackage{graphicx}

\numberwithin{equation}{section}

\newtheoremstyle{fancy1}{10pt}{10pt}{\itshape}{12pt}{\textsc\bgroup}{.\egroup}{8pt}{
}
\newtheoremstyle{fancy2}{10pt}{10pt}{}{12pt}{\itshape}{.}{8pt}{ }

\theoremstyle{fancy1}

\newtheorem{cor}[equation]{Corollary}
\newtheorem{lem}[equation]{Lemma}
\newtheorem{prop}[equation]{Proposition}
\newtheorem{thm}[equation]{Theorem}

\newtheorem{main}{Theorem}
\newtheorem*{main*}{Theorem}

\newtheorem*{cor*}{Corollary}

\theoremstyle{fancy2}

\newtheorem{rem}[equation]{Remark}
\newtheorem*{rem*}{Remark}
\newtheorem{example}[equation]{Example}

\newcommand{\cref}[1]{Corollary~\ref{#1}}

\newcommand{\Sph}{\mathbb{S}}

\newcommand{\Q}{\mathsf{Q}}

\newcommand{\N}{\mathsf{N}}

\newcommand{\C}{{\mathbb{C}}}
\newcommand{\R}{{\mathbb{R}}}
\newcommand{\Z}{{\mathbb{Z}}}

\newcommand{\E}{\ensuremath{\operatorname{\mathsf{E}}}}
\newcommand{\F}{\ensuremath{\operatorname{\mathsf{F}}}}
\newcommand{\G}{\ensuremath{\operatorname{\mathsf{G}}}}

\newcommand{\SO}{\ensuremath{\operatorname{\mathsf{SO}}}}

\newcommand{\Sp}{\ensuremath{\operatorname{\mathsf{Sp}}}}

\newcommand{\SU}{\ensuremath{\operatorname{\mathsf{SU}}}}
\newcommand{\Spin}{\ensuremath{\operatorname{\mathsf{Spin}}}}

\newcommand{\T}{\ensuremath{\operatorname{\mathsf{T}}}}
\renewcommand{\S}{\ensuremath{\operatorname{\mathsf{S}}}}

\newcommand{\A}{\ensuremath{\operatorname{\mathsf{A}}}}

\def\con#1=#2(#3){#1 \equiv #2 \bmod{#3}}

\renewcommand{\F}{\mathsf{F}}

\begin{document}

\title{Dual submanifolds in rational homology spheres}

\author{Fuquan Fang}
\address{Department of Mathematics\\Capital Normal University\\
Beijing 100048\\China\\ } \email{fuquan\_fang@yahoo.com }

\thanks{ The author is supported by a NSFC key grant, the Ministry of Education in China, and the municipal administration of Beijing.}

\maketitle

\vskip 4mm

\centerline{{\it Dedicated to Professor Boju Jiang on the occasion of his 80th birthday}}

\vskip 4mm



\begin{abstract}
Let $\Sigma$ be a simply connected rational homology sphere. A pair of disjoint closed submanifolds $M_+, M_-\subset \Sigma$ are called {\it dual} to each other if the complement $\Sigma - M_+$ strongly homotopy retracts onto $M_-$ or vice-versa. In this paper we are concerned with the basic problem of which integral triples $(n; m_+, m_-)\in \Bbb N^3$ can appear, where $n=\text{dim} \Sigma -1$, $m_\pm=\text{codim}M_\pm -1$. The problem is motivated by several  fundamental aspects in differential geometry:

(i) the theory of isoparametric hypersurfaces and Dupin hypersurfaces in the unit sphere $\Bbb S^{n+1}$ initiated by \'Elie Cartan, where $M_\pm$ are the focal manifolds of the hypersurface $M\subset \Bbb S^{n+1}$, and $m_\pm$ coincide with the multiplicities of principal curvatures of $M$.

(ii) the Grove-Ziller construction of non-negatively curved riemannian metrics on exotic spheres, where $M_\pm$ are the singular orbits of a cohomogeneity one action on $\Sigma$.

Based on important result of Grove-Halperin \cite{GH}, we provide a surprisingly simple answer, namely, if and only if one of the following holds true:

$\bullet$ $m_+=m_-=n$,

$\bullet$ $m_+=m_-=\frac 13 n \in \{1, 2, 4, 8\},$

$\bullet$ $m_+=m_-=\frac 14 n \in \{1, 2\},$

$\bullet$ $m_+=m_-=\frac 16 n \in \{1, 2\},$

$\bullet$ $\frac n{m_++m_-}=1$ or $2$, and for the latter case, $m_++m_-$ is odd if $\text{min}(m_+,m_-)\geq 2$.

In addition, if $\Sigma$ is a homotopy sphere and the ratio $\frac n{m_++m_-}=2$ (for simplicity let us assume $2\leq m_-<m_+$),  we observe that, the proof in Stolz \cite{St} applies almost identically to conclude that, the pair can be realized if and only if, either $(m_+, m_-)=(5,4)$ or $m_++m_-+1$ is divisible by the integer $\delta(m_--1)$ (cf. the table on page 2), which is equivalent to the existence of $(m_--1)$ linearly independent vector fields on the sphere $\Sph^{m_++m_-}$ by Adams' celebrated work. In contrast, infinitely many counterexamples are given if $\Sigma$ is a rational homology sphere.
\end{abstract}

\bigskip


\section{Introduction}

Let $\Sigma$ be a simply connected rational homology sphere of dimension $n+1$, and let $M_+\sqcup M_-\subset \Sigma$ be an embedded closed submanifold with two connected components. We call $M_+$ is {\it dual} to $M_-$ if the complement  $\Sigma - M_+$ is strongly homotopy retracts onto $M_-\subset \Sigma$. It is clear that the dual relation is reflective, i.e., if $M_+$ is dual to $M_-$, then $M_-$ is dual to $M_+$. Moreover, note that $$\Sigma= \Bbb D(\nu_+)\cup _\partial \Bbb D(\nu_-)$$
where $\nu_\pm$ is  the normal disk bundles of $M_\pm$ in $\Sigma $.

Let $m_\pm$ denote the dimensions of the normal spheres to $M_\pm \subset \Sigma $. In this paper we are concerned with the following

\smallskip

{\bf Problem.} {\it Which integral triples $(n; m_+, m_-)\in \Bbb N^3$ can be realized  as the dimension of $M$, codimensions of $M_+$ and $M_-$?}

Besides the interests in its own rights from algebraic topology,  the problem has roots in at least two important themes in differential geometry:

$\bullet$ the theory of isoparametric hypersurfaces or more generally of Dupin hypersurfaces in the unit sphere $\Bbb S^{n+1}$ initiated by \'Elie Cartan in  \cite{Ca1}\cite{Ca2}\cite{Ca3}\cite{Ca4},  where $M_\pm$ are the focal manifolds of the hypersurface $M\subset \Bbb S^{n+1}$, and $m_\pm$ are referred as the multiplicities of principal curvatures of $M$.

 $\bullet$  cohomogeneity one isometric actions on a homology sphere $\Sigma$ where $M_\pm$ are the singular orbits, which produce many important examples in riemannian geometry including Einstein metrics, minimal submanifolds (cf. Hsiang-Lawson \cite{HL}), and very recently, new examples of non-negatively curved riemannian metrics on exotic $7$-spheres by the Grove-Ziller \cite{GZ}.

It is clear that, given a trivial $k$-knot $\Sph^k\subset \Sph^{k+\ell+1}$, there is a {\it dual} trivial $\ell$-knot $\Sph^\ell \subset \Sph^{k+\ell+1}$. A nontrivial dual pair are the embedded real projective planes  $\Bbb{RP}_\pm^2$ in $\Sph^4$, for  which the normal $\Sph^1$-bundles of $\Bbb{RP}_\pm^2$ have the same total space $\Sph^3/\Bbb Q_8$ where $\Bbb Q_8=\{\pm 1, \pm i, \pm j, \pm  k\}\subset \S^3$ the quaternion subgroup. Both examples arise in the theory of isoparametric hypersurfaces in the unit spheres, i.e., the principal curvatures are constant. From the classification of \'Elie Cartan \cite{Ca1}\cite{Ca2}\cite{Ca3}\cite{Ca4}, the focal manifold of any isoparametric hypersurface in the unit sphere  with two distinct principal curvatures is the union of a dual pair of trivial knots, moreover, isoparametric hypersurface with three distinct principal curvatures only occurs in the unit spheres of dimensions $4, 7, 13$ and respectively $ 25$, whose focal manifold is $\Bbb {FP}_\pm^2$ with $\Bbb F=\Bbb R, \Bbb C, \Bbb H$ and respectively  $\F_4/\Spin(9)_\pm$ where the multiplicities $m_+=m_-=1,2, 4, 8$. In general, a celebrated  theorem of M\"unzner \cite{Mu} shows that the number $g$ of distinct principal curvatures of an isoparametric
hypersurface $M\subset  \Bbb S^{n+1}$ must be $1$, $2$, $3$, $4$ or $6$. Furthermore, when $g$ is odd, the multiplicities of the $g$ principal curvatures are all the same,  $m: =m_1=m_2=\cdots =m_{g}$, and  when $g$ is even, we have equalities $m_1=m_3=\cdots =m_{g-1}: =m_+$ and $m_2=m_4=\cdots =m_g: =m_-$. From definition the dimension of $M$ and its multiplicities satisfy the formula $2n=g(m_{-}+m_{+})$.

\vskip 1mm

In very recent two decades, breakthroughs have been made for the classification of isoparametric hypersurfaces in the spheres in
the series \cite{CCJ1}\cite{CCJ2}\cite{DN}\cite{Miy} answered in positive an open problem in the list of S.T.Yau \cite{Ya}, it turns out that, an isoparametric hypersurface in the unit sphere is either homogeneous or one of the Ferus-Karcher-M\"unzner examples from the representations of Clifford algebra. More precisely, the homogeneous isoparametric hypersurfaces are the principal orbits of a cohomogeneity one isometric action on $\Sph^{n+1}$  where the focal manifolds are the two singular orbits. The  Ferus-Karcher-M\"unzner
\cite{FKM} examples are constructed using the
orthogonal representations of Clifford algebra as follows:  Let
$Cl^{0,m+1}$ denote the Clifford algebra spanned by $1,e_0,\cdots ,e_m$
satisfying $e_i^2=1$ and $e_ie_j=-e_je_i$ for $i\neq j$.  For any
nontrivial  $(n+2)$-dimensional orthogonal representation of
$Cl^{0,m+1}$, $e_0,\cdots ,e_m$ give rise matrices $P_0,\cdots ,P_m$
satisfying that $P_i^2=I$ and $P_iP_j=-P_jP_i$ for $i\neq j$. Let
\vskip 1mm

\centerline {$f(x)=\langle x,x\rangle _{}^2-2\sum_{i=0}^{m_{}}\langle
P_i(x),x\rangle ^2$,
$x\in \R ^{n+2}.$ }
\vskip 3pt

The function $f$ maps the unit sphere to $[-1,1]$ and satisfies the
Cartan-M\"unzner equations, i.e, $\Vert df\Vert ^2$ and the Laplacian
$\Delta f$ are both functions of $f$. By \cite{Mu}, for
any regular value $c\in [-1,1]$,
the hypersurface $M=f^{-1}(c)$ is  an isoparametric hypersurface with
four distinct principal curvatures and multiplicities $m$, $\frac n2-m$,
$m$,  $\frac n2-m$. Its scalar curvature is constant and equal to $n^2-4n$.
Notice that all irreducible representations of $Cl^{0,m+1}$ have the same
dimension $\ 2\delta (m)$ where $\delta (m),$ $m\geq 1$ is an integral
function satisfying that $\delta (m+8k)=2^{4k}\delta (m)$ and
\vskip 2mm

\centerline{\begin{tabular}{|r|c|c|c|c|c|c|c|c|} \hline
\noindent $m$ & $ 1$ &
$2 $ & $3$  & $4$ & $5$ &$6$ & $7$ & $8$
\\ \hline
\noindent $\delta (m)$ &  $1$
& $2 $ & $4$  & $4$ & $8$ &$8$
& $8$ & $8$ \\ \hline
\end{tabular} }
\vskip 2mm

\noindent From the construction we know that $(n+2)$ must be a multiple of
$2\delta (m)$. We refer to \cite{FKM} for more details about those examples.

On the other hand, given a cohomogeneity one action of  a connected compact Lie group $\G$ on a simply connected rational homology sphere $\Sigma$, the orbit space $\Sigma/\G$ is isometric to an interval $[-1,1]$, hence there are exactly two singular orbits $M_\pm$ corresponding to $\pm 1\in [-1,1]$. It is clear that $M_+$ and $M_-$ are dual pair in $\Sigma$. For a linear (representation) cohomogeneity one action on the sphere $\Sph^{n+1}$, the principal orbits are isoparametric hypersurfaces, and the singular orbits are the focal manifolds.  However, there are infinitely many nonlinear cohomogeneous one  actions on spheres. Recently, Grove-Ziller \cite{GZ} constructed non-negatively curved riemannian metrics on the Milnor spheres $\Sigma$ (homotopy $7$ spheres which are $\Sph^3$ bundles over $\Sph^4$) via cohomogeneity  one actions  whose the singular orbits are of codimensions $2$, i.e., $m_\pm=1$. Very recently, Goette-Kerin-Shankar has announced a generalized result along the line of Grove-Ziller that all $7$-dimensional exotic spheres admit riemannian metrics with non-negative sectional curvature, once again there is a pair of dual submanifolds in
$\Sigma$ of codimension $2$. So far there is no any example of exotic spheres of dimension $>7$ admitting riemannian metrics with non-negative sectional curvature. Observe that the method of Grove-Ziller depends heavily on the codimension $2$ property of the dual submanifolds. It is natural to ask, if there are homotopy spheres of dimension $>7$ admitting a dual pair of codimension $2$, i.e., $m_+=m_-=1$.

Our main result for this is the following

\begin{main}
Let $\Sigma$ be a simply connected rational homology sphere of dimension $n+1$ and let $M_\pm \subset \Sigma$ be a pair of dual submanifolds of codimensions $m_{\pm}+1$. Then one of the following holds:

$\bullet$ $m_+=m_-=n$,

$\bullet$ $m_+=m_-=\frac 13 n \in \{1, 2, 4, 8\},$

$\bullet$ $m_+=m_-=\frac 14 n \in \{1, 2\},$

$\bullet$ $m_+=m_-=\frac 16 n \in \{1, 2\},$

$\bullet$ $\frac n{m_++m_-}=1$ or $2$, and for the latter case, $m_++m_-$ is odd if $\text{min}(m_+,m_-)\geq 2$.
\end{main}

We remark that Theorem A implies immediately the celebrated theorem of  M\"unzner, that the number $g \in \{1,2, 3, 4, 6\}$ of distinct principal curvatures for an isoparametric hypersurface, by the formula $2n=g(m_++m_-)$ when $g$ is even, and $n=gm_+=gm_-$ when $g$ is odd.

It is also clear to read from Theorem A that

\begin{cor} Let $\Sigma$ be a simply connected rational homology sphere of dimension $n+1$ with dual submanifolds $M_\pm\subset \Sigma$. If $m_+=m_-=1$, then $n\in \{1, 2, 3, 4, 6\}$.
\end{cor}

\begin{cor} There is no riemannian submersion $\pi: M\to \Sigma$ from a simply connected cohomogeneity one manifold $M$  whose singular orbits are of codimension $2$ onto a rational homology sphere $\Sigma$ of dimension at least $8$.
\end{cor}

Therefore, it seems that the Grove-Ziller construction does not produce riemannian metrics on rational homology spheres, in particular exotic spheres of dimension at least $8$ with non-negative sectional curvature.

Note that, when $n=m_++m_-$, any pair $(m_+, m_-)$ can be realized as the trivial dual knots in the sphere. When $\frac n{m_++m_-}=2$ and $m_+\neq m_-$, it is very difficult to determine which pair of integers can be realized.  Without loss of generality, we may assume that $m_-< m_+$, in the special case that $\Sigma$ is in addition a homotopy sphere,  though the Stolz's theorem was formulated in riemannian geometry (cf. Theorem 4.1), we observe that the proof of Stolz's theorem \cite{St} applies almost identically, to prove the following pure topological theorem, we tribute to Stephan Stolz.

\begin{thm} [Stolz]
Let $\Sigma$ be a homotopy sphere of dimension $n+1$, and let $M_\pm \subset \Sigma$ be a pair of dual submanifolds codimensions $m_{\pm}+1$ where $m_-< m_+$. If $n=2(m_++m_-)$, then either  $(m_+, m_-)=(5,4)$, or $m_++m_-+1$ is divisible by $\delta(m_--1)$.
\end{thm}

\begin{rem}
It is natural to ask whether a similar result as above holds true for rational homology sphere. We will provide infinitely many examples of dual submanifolds in rational homology spheres of dimension $4m-1$, such that any pair of integers $(m_+, m_-)$ so that $m_++m_-=2m-1$ can be realized as the dimensions of the normal spheres to $M_\pm$.
\end{rem}

\begin{rem}
Though Theorem 1.3 does not hold for rational homology spheres in general, it is still interesting to wonder, under what constraints on $\Sigma$, Theorem 1.3 holds true.  The proof of the Stolz's theorem \cite{St}  depends heavily the cell structure of the hypersurfaces $M$ and the focal manifolds $M_\pm$  which seems to be impossible to  generalize.   In \cite{Fr}, ahead the work of Stolz \cite{St}, the $K$-theory of isoparametric hypersurfaces was developed  which solved in half cases of the multiplicities problem. It might be useful to apply $K$-theory for the problem.
\end{rem}

In view of the above results it follows that, for a pair of dual submanifolds $M_+ \sqcup M_-\subset \Sigma$ with the dimension data $(n;m_+,m_-)$  where $\Sigma$ is a homotopy sphere,  there is an isoparametric hypersurface $N\subset \Sph^{n+1}$ with focal submanifold $N_+\sqcup N_- \subset \Sph^{n+1}$ with the same dimensions data. Moreover, by Tables 2.2 and 2.3 the fundamental groups and homology groups of $N_\pm$ coincide with that of $M_\pm$. We conclude this section with the following  natural but difficult problem which  is already highly nontrivial and interesting when $M$ is a Dupin hypersurface in the sphere (cf. \cite{F2}).

\vskip 2mm

{\bf Problem.} Is every pair of dual submanifolds $M_+ \sqcup M_-\subset \Sigma$ in a homotopy sphere topologically homeomorphic to the pair of focal manifolds of an isoparametric hypersurface in a sphere?

\vskip 2mm

{\bf Acknowledgement}. The author would like to thank Karsten Grove for useful discussions motivated the corollaries in the paper.

\section{Preliminaries}

In this section we present some basic preliminary results of Grove-Halperin \cite{GH}.

\subsection{Double mapping cylinder and rational homotopy theory} In \cite{GH} the authors consider the so called {\it double mapping cylinder} $DE$ of maps $\phi_\pm: E\to B_\pm$ with homotopy fibers of $\Sph^{m_\pm}$ up to weak homotopy equivalence, i.e., the gluing of the mapping cylinders of $\phi_{\pm}$ along $E$. Let $F$ denote a path connected component of a homotopy fiber of the inclusion $j:  E\to DE$.
For simplicity let us assume that

$\bullet$ $E, B_\pm$ are connected CW complexes, $DE$ is simply connected, and the homotopy fibres $\Sph^{m_\pm}$ satisfy $m_\pm\geq 1$.

For simplicity we may assume that $m_+\geq m_-\geq 1$. If $m_-=1$, then $[\phi_+(\Sph^{m_-})]\in \pi_1(B_+)$ acts on the homology group
$H_{m_+}(\Sph^{m_+};\Bbb Z)$ of the homotopy fibers of $\phi_+: E\to B_+$.  By \cite{GH}, $\phi_+$ is called {\it twisted }if this action is non trivial, hence by  $-1$. Similarly, $\phi_-$ is {\it twisted} if $m_+=1$ and $[\phi_-(\Sph^{m_+})]\in \pi_1(B_-)$ acts by $-1$ on the homology group
$H_{m_-}(\Sph^{m_-};\Bbb Z)$ of the homotopy fibers of $\phi_-: E\to B_-$.

\
Under this assumption, we have the following important result in \cite{GH}.

\begin{thm}[Grove-Halperin] The fundamental group $\pi_1(F)$ and the homology group $H_*(F;\Bbb Z)$ are given in the following tables
where $\Q=\{\pm 1, \pm i, \pm j, \pm k\} \subset \S^3$ is the order $8$ quaternion group.
\end{thm}

     {\setlength{\tabcolsep}{0.10cm}
\renewcommand{\arraystretch}{1.6}
\stepcounter{equation}
\begin{table}[!h]\label{exceptional}
      \begin{center}
       \begin{tabular}{|c||c|c|c|c|c|}

\hline
   $(m_+, m_-) $ & $m_\pm >1$ & $m_+>m_-=1$ &  $m_+=m_-=1$      & $m_+=m_-=1$   & $m_+=m_-=1$   \\
   & &  & no $\phi_\pm$ twisted & one of $\phi_\pm$ twisted & both  $\phi_\pm$ twisted  \\

\hline

   $\pi_1(F)$ & $1$ & $\Bbb Z$ & $\Bbb Z\oplus \Bbb Z$  & $\Bbb Z\oplus \Bbb Z_2$  & $\Q$   \\

\hline

         \end{tabular}
      \end{center}

         \caption{$\pi_1(F)$} \label{fund}

\end{table}}

     {\setlength{\tabcolsep}{0.10cm}
\renewcommand{\arraystretch}{1.6}
\stepcounter{equation}
\begin{table}[!h]\label{exceptional}
      \begin{center}
       \begin{tabular}{|c||c|l|}

\hline
   $(m_+, m_-) \mbox{ and  twists} $ & $H_i(F;\Bbb Z)$   & $i$ \\

\hline

 $m_+\neq m_-$      & $\Bbb Z$  &   $i=0$ or $i=m_+, m_-\text{ mod }(m_++m_-)$     \\

no twists & $\Bbb Z \oplus \Bbb Z$ &     $i>0$ and $i=0 \text{ mod }(m_++m_-)$ \\ \hline

$m_+=m_-$   & $\Bbb Z$  & $i=0$ \\

no twists  & $\Bbb Z \oplus \Bbb Z$ &  $i>0$ and $i=0 \text{ mod }(m_+)$ \\ \hline

$m_+>m_-=1$      & $\Bbb Z$  &  $i=0$ or $i=\pm 1  \text{ mod }(2m_++2)$  \\
$\phi_+$ twisted  &  $\Bbb Z \oplus \Bbb Z$ & $i>0$ and $i=0\text{ mod }(2m_++2)$  \\
&  $\Bbb Z_2$ &
$i=m_+$,   $m_++1\text{ mod }(2m_++2)$
\\ \hline

$m_+=m_-=1$;     &  $\Bbb Z$  & $i=0$ or $i=3 \text{ mod }(4)$  \\
 $\phi_+$ twisted  & $\Bbb Z \oplus \Bbb Z_2$ & $i=1 \text{ mod }(4)$\\
 $\phi_-$ not twisted  & $\Bbb Z_2$ & $i=2 \text{ mod }(4)$\\
&  $\Bbb Z \oplus \Bbb Z $ & $i>0$ and $i=0 \text{ mod }(4)$\\ \hline

$m_+=m_-=1$      & $\Bbb Z$ & $i=0$ \\
$\phi_\pm$ both twisted  & $\Bbb Z \oplus \Bbb Z$ & $i>0$ and $i=0 \text{ mod }(3)$\\

&    $\Bbb Z_2 \oplus \Bbb Z_2$ & $i=1 \text{ mod }(3)$
\\ \hline

         \end{tabular}
      \end{center}

      \vspace{0.1cm}
         \caption{$H_*(F;\Bbb Z)$ } \label{rational}

\end{table}}

\newpage

\begin{thm}[Grove-Halperin] The rational homotopy type of $F$ is given in the following table. Moreover, the exceptional cases $\A_4(4)\times \Omega \Sph^{17}$,   $\A_6(4)  \times \Omega \Sph^{25}$ do not occur if $DE$ is a homotopy sphere and $\phi _\pm$ are normal sphere bundles of $B_\pm$.
\end{thm}

     {\setlength{\tabcolsep}{0.10cm}
\renewcommand{\arraystretch}{1.6}
\stepcounter{equation}
\begin{table}[!h]\label{exceptional}
      \begin{center}
       \begin{tabular}{|c||c|}

\hline
   $(m_+, m_-) \mbox{ and  twists} $ & $ \Bbb  Q$ \mbox{ homotopy type of } $F$  \\

\hline

 $m_+=m_-=1$;      & $[\SO(3)/(\Bbb Z_2\oplus \Bbb Z_2)]\times \Omega \Sph^4  $     \\

 $\phi_\pm$  both twisted & $\simeq _\Bbb Q [\SO(4)/(\Bbb Z_2\oplus \Bbb Z_2)]\times \Omega \Sph^7  $    \\ \hline

$m_+=m_-=1$;   &  \\

$\phi_+$ twisted, not $\phi_-$  & $[(\SO(2)\times \SO(3))/\Bbb Z_2]\times \Omega \Sph^5  $  \\ \hline

$m_+=m_-=1$;      & $\Bbb S^1\times \Bbb S^1\times \Omega \Sph^3$      \\
$\phi_\pm$  both not twisted  &  $\Bbb S^1\times \Omega \Sph^2$ \\ \hline

$m_+>m_-=1$;     &  \\
 $m_+$ odd, $\phi_+$ twisted  & $\Bbb S^1\times \Bbb S^{2m_++1}\times \Omega \Sph^{2m_++3}$

\\ \hline

$m_+>m_-=1$;       & $\Bbb S^1\times \Bbb S^{m_+}\times \Omega \Sph^{m_++2}$     \\

 $\phi_+$ not twisted &  $\simeq _{\Bbb Q} \Bbb S^1\times \Bbb S^{m_+}\times \Bbb S^{m_++1}\times \Omega \Sph^{2m_++3}$ if $m_+=0(\text{mod }2)$
\\ \hline

$m_+>m_-\geq 2$   & $\Bbb S^{m_+}\times \Bbb S^{m_-}\times \Omega \Sph^{m_++m_-+1}$     \\   \hline

$m_+=m_-$ odd     & $\Bbb S^{m_+}\times \Bbb S^{m_+}\times \Omega \Sph^{2m_++1}$   \\

&  $\simeq _{\Bbb Q} \Bbb S^{m_+}\times \Omega \Sph^{m_++1}$
\\ \hline

$m_+=m_-$ even      & $\Bbb S^{m_+}\times \Bbb S^{m_+}\times \Omega \Sph^{2m_++1}$ \\

&   $\Bbb S^{m_+}\times \Omega \Sph^{m_++1}$     \\   \hline

$m_+=m_-=2$       &  $\SU(3)/\T^2  \times \Omega \Sph^{7}$;\\
 & $\Sp(2)/\T^2 \times \Omega \Sph^{9}$ \\
 &  $\G_2/\T^2  \times \Omega \Sph^{13}$  \\ \hline

$m_+=m_-=4$       &
$\Sp(3)/\Sp(1)^3 \times \Omega \Sph^{13}$\\
&  $\A_4(4)\times \Omega \Sph^{17}$ \\
&  $\A_6(4)  \times \Omega \Sph^{25}$
 \\   \hline

$m_+=m_-=8$       & $\F_4/\Spin(8) \times \Omega \Sph^{25}$ \\   \hline

         \end{tabular}
      \end{center}

      \vspace{0.1cm}
         \caption{Rational homotopy type of $F$} \label{rational}

\end{table}}

\newpage

Note that the cohomology ring
$$H^*(\Omega \Sph^{2k-1};\Bbb Q)=\Bbb Q [x]$$
the free polynomial algebra, where $x$ is a degree $2k-2$ generator. Moreover, the rational homotopy type $\Omega \Sph^{2k}\simeq _\Bbb Q \Sph^{2k-1}\times \Omega \Sph^{4k-1}$. Let $\A_m(k)$ be the simply connected space ($k$ even, $m=1, 2, 3, 4, 6$) unique up to rational homotopy type, with cohomology algebra
$H^*(\A_m(k); \Bbb Q)\cong \Bbb Q[x,y]$ where $x, y$ are of degree $k$ subject to relations
$$
\begin{array}{lll}
x^m=x^2+y^2=0 & \mbox{ if } m=1, 2, 4.\\
x^m=x^2+3y^2=0 & \mbox{ if } m=3,6.\\
\end{array}
$$

Note that $$\A_1(k)\simeq _\Bbb Q \Sph^k \mbox{  ;   } \A_2(k)\simeq _\Bbb Q \Sph^k \times \Sph^k$$
and when $m=3, 4, 6$ then $k=2, 4, 8$. Moreover,
$$\SU(3)/\T^2\simeq _\Bbb Q \A_3(2)\mbox{  ;   }  \Sp(2)/\T^2\simeq _\Bbb Q \A_4(2)\mbox{  ;   }  \G_2/\T^2\simeq _\Bbb Q \A_6(2)$$

$$\Sp(3)/\Sp(1)^3 \simeq _\Bbb Q \A_3(4)\mbox{  ;   }  \F_4/\Spin(8) \simeq _\Bbb Q \A_3(8) $$

\section{Proof of Theorem A}

If either of $m_\pm$ equals $n$,  the corresponding manifold $M_\pm$ is a point, hence the dual one is also a point and so $m_+=m_-=n$. In this case $\Sigma$ is forced to be a homotopy sphere. In the following we assume $m_\pm <n$.

By definition  $H_*(\Sigma; \Bbb Q)\cong H_*(\Sph^{n+1}; \Bbb Q)$ since $\Sigma$ is a rational homology sphere $\Sigma$. The degree one map $f: \Sigma \to \Sph^{n+1}$ is a rational homotopy equivalence.  Let $i: M\subset \Sigma$ be the hypersurface and let $F$ denote the homotopy fiber of the inclusion. Consider the homotopy fibrations
$$\Omega \Sigma  \to F\to M \to  \Sigma $$
where $\Omega \Sigma$ is the based loop space of $\Sigma$.

For simplicity we will often use $\simeq_\Bbb Q$ to denote rational homotopy equivalence.  Since $f\circ i: M\to \Sigma \to \Sph^{n+1}$ is contractible, it follows that $F\simeq _\Bbb Q \Omega \Sigma \times M\simeq _\Bbb Q\Omega \Sph^{n+1}  \times M$. Therefore, the cohomology rings  $$H^*(F; \Bbb Q)\cong H^*(\Omega \Sph^{n+1} ; \Bbb Q)\otimes H^*(M; \Bbb Q). $$
Note that, if $n$ is even, then $H^*(\Omega \Sph^{n+1}; \Bbb Q)\cong \Bbb Q [e_{n}]$ is a polynomial ring on a variable $e_{n}$ of degree $n$; and if $n$ is odd, then $H^*(\Omega \Sph^{n+1}; \Bbb Q)\cong \Bbb Q [e_{2n+1}]\otimes \E (e_{n})$ where $\E (e_{n})$ is the exterior algebra on $e_n$. Indeed, it is well known that $\Omega \Sph^{n+1}\simeq _\Bbb Q
\Sph^n\times \Omega \Sph^{2n+1} $ when $n$ is odd.

On the other hand, by Theorem 2.4 it follows that the rational homotopy type of $F$ is either of the form $\Sph^k\times \Sph^l \times  \Omega \Sph^{k+l+1} $ or of the form $\A_m(k)\times  \Omega \Sph^{mk+1}$ with $m=1, 2, 3, 4 , 6$ and $k$ even, where $\A_1(k)\simeq _\Bbb Q \Sph^k \mbox{  ;   } \A_2(k)\simeq _\Bbb Q \Sph^k \times \Sph^k$. Moreover, if $m=3, 4, 6$, then $k=2, 4, $ or $8$. Now we compare this with the previous rational equivalence $F\simeq _\Bbb Q \Omega \Sph^{n+1}  \times M$.  The point of departure is

\begin{lem} If $\Omega \Sph^{n+1}  \times X \simeq _\Bbb Q \Omega \Sph^{\ell +1}  \times Y$, where $X$ and $Y$ are finite CW complexes, then $m=n$ if $m, n$ have the same pairities. Moreover, if $n$ is odd but $m$ is even, then $m=2n$.
\end{lem}
\begin{proof} Recall that the loop space $\Omega \Sph^{n+1} $ (resp. $\Omega \Sph^{\ell +1}$) contains a factor of the form $\Omega \Sph^{2i+1}$ no matter $n+1$ is odd or even. The cohomology ring $H^*(\Omega \Sph^{2i+1}; \Bbb Q)$ is a free polynomial algebra on a generator $e_{2i}$ of degree $2i$. The cohomology ring $H^*(\Omega \Sph^{n+1}  \times X;\Bbb Q)=H^*(\Omega \Sph^{n+1} ;\Bbb Q) \otimes H^*(X;\Bbb Q)$ where $H^*(X;\Bbb Q)$ is of finite dimensional. Therefore, the factors of free polynomial algebras of the isomorphism
$ H^*(\Omega \Sph^{n+1}  \times X,  \Bbb Q) \cong H^*( \Omega \Sph^{\ell +1}  \times Y, \Bbb Q)$ must be of  the same degree.   The desired result follows.
\end{proof}

If $F\simeq _\Bbb Q \Sph^k\times \Sph^l \times  \Omega \Sph^{k+l+1} $, it is clear to note that $k+l\leq n$. By the above lemma we get that, either $n=k+l$ or $n=2(k+l)$. By Table 2.5 it follows that, if $m_+>m_-\geq 2$, then $(k,l)=(m_+, m_-)$, and therefore $\frac n{m_++m_-}\in \{1, 2\}$. Moreover, if $m_+>m_-\geq 2$ and in addition, $m_++m_-$ is even, then  $\frac n{m_++m_-}=1$.
It is straightforward to check that, if either $m_+>m_-=1$ or   $m_+=m_-=1$, then either $k+l=(m_++m_-)$, $ 2(m_++m_-)$ or  $3(m_++m_-)$. In particular, if $m_+=m_-=1$, then $n=2, 4, 6$.

If $F\simeq_\Bbb Q\A_m(k)\times  \Omega \Sph^{mk+1}$ with $m \in \{1, 2, 3, 4, 6\}$ and $k$ even. Since $A_1(k)\simeq _\Bbb Q \Sph^k$ and $A_2(k)\simeq _\Bbb Q \Sph^k\times \Sph^k$ and so $F$ is of the form in the previous case, we may assume that $m\in \{3, 4, 6\}$, then $m_+=m_- =k\in \{2, 4, 8\}$. By using Lemma 3.1 again it follows that, $n=mk$ if $n$ and $mk$ have the same pairity, otherwise, either $n=2mk$ or  $2n=mk$, the latter can not occur for dimensional reasoning. For the former, since $mk$ is always even, hence $n=2mk$ implies that $n$ is also even and has the same pairity as $mk$, a contradiction. Therefore, $n=mk$ with $k\in \{2, 4, 8\}$ and $m\in \{3, 4, 6\}$. By Table 2.5 it suffices to exclude the
cases $A_4(4)\times \Omega \Sph^{17}$ and $A_6(4)\times \Omega \Sph^{25}$, where $\Sigma$ is a simply connected rational homology sphere of dimension $17$ and respectively $25$.

\smallskip
The following lemma is probably well-known to experts.

\begin{lem} Let $\gamma$ denote the universal $\Sph^4$-bundle $p: B\SO(4)\to B\SO(5)$. Then the Euler class $e(\gamma)=0$.
\end{lem}
\begin{proof} Note that in general the Euler class of an odd dimensional vector bundle is an order $2$ element in the cohomology group with integer coefficient. It is well-known (from the definition) that $e(\gamma)$ can be calculated from
the Gysin exact sequence with integer coefficents
$$H^4(B\SO(5))\stackrel{p^*    }{
\longrightarrow } H^4(B\SO(4)) \to H^0(B\SO(5))  \stackrel{\cup e(\gamma)  }{
\longrightarrow }    H^5 (B\SO(5))
\stackrel{p^*    }{
\longrightarrow }  H^5 (B\SO(4))\to 0$$
The final two terms $H^5 (B\SO(4))$ and $H^5 (B\SO(5))$ can be calculated from the Serre spectral sequences with integer coefficients of the homotopy fibration $B\Spin(4) \to B\SO(4)\to K(\Bbb Z_2, 2)$  and the homotopy fibration
$B\Spin(5)  \to B\SO(5)\to K(\Bbb Z_2, 2)$. Note that $B\Spin(4)=B\S^3\times B\S^3$ and $B\Spin(5)=B\Sp(2)$.   It is routine to see that only the $E_2$-term $E_2^{5, 0}=$

\noindent $=H^5( K(\Bbb Z_2, 2); H^0(B\Spin(4)))\cong \Bbb Z_2$ and respectively $E_2^{5, 0}=H^5( K(\Bbb Z_2, 2); H^0(B\Sp(2)))\cong \Bbb Z_2$ survives contributing to $H^5(B\SO(4))$ and respectively  $H^5(B\SO(5))$. Therefore, the last homomorphism $p^*$ in the above Gysin sequence is an isomorphism,   and it follows that $e(\gamma)=0$.
\end{proof}

By the above Lemma it follows that the Euler classes the oriented $\Sph^4$ bundles $\pi_\pm: M \to M_\pm$ are zero, hence the long exact sequences  (up to degree $5$) of the $\Sph^4$-bundles implies that there are cohomology classes $\alpha ^\pm \in H^4(M)$ such that the restrictions of $\alpha^\pm$ on the fibers of $\pi_\pm$ are generators of the cohomology groups $H^4(\Sph^4)$ (different fibers).  From the Leray-Hirsch Lemma it follows that $$H^*(M)\cong H^*(M_+)[1, \alpha ^+] \cong H^*(M_-)[1, \alpha ^-] \hspace{2cm} \star$$ as free modules.

On the other hand, since $F$ is $3$-connected, $\Sigma$ is simply connected, there is an exact sequence up to degree $5$ from the homotopy fibration  $F\to M\to \Sigma$
$$H^4(\Sigma)\to H^4(M)\to H^4(F) \cong \Bbb Z\oplus \Bbb Z=\langle \alpha^+, \alpha^-\rangle$$
where $H^4(\Sigma)$ is a torsion group since $\Sigma$ is a rational homology sphere, it follows that the torsion free part of $H^4(M)$, i.e., modulo torsion, $H^4(M)/T=\langle \alpha^+, \alpha^-\rangle$. In particular, $\alpha^+, \alpha^-$ are linearly independent. This together with  $\star$ implies that $\alpha ^+$ lies in the image of $\pi _- ^*: H^4(M_-)\to H^4(M)$ and $\alpha ^-$ lies in the image of $\pi _+ ^*: H^4(M_+)\to H^4(M)$.

Now we need to derive the multiplicative structure of $H^*(M)/T$, where $T$ is the torsion. Let $\alpha ^{+}_i, \alpha ^{-}_i \in H^{4i}(M)  $ denote the generator of free part (isomorphic to $\Bbb Z$) of the image of $\pi_\pm ^*: H^{4i}(M_\pm) \to H^{4i}(M)$ (in particular, $\alpha^\pm=\alpha^\pm_1$).

For the boundary homomorphism $\partial^* : H^*(M)\to H^{*+1}(\Sigma)$ in the Mayer-Vietories exact sequence of $(\Sigma; \Bbb D(\gamma_+), \Bbb D(\gamma_-))$, note that $\partial^*(\alpha \beta)=\partial^*(\alpha)\beta +\alpha \partial^*(\beta)$ when $\alpha, \beta$ are both of even degree.   Therefore, $\partial^*(\alpha^+\alpha^-)=0$, since $\partial^*(\alpha^+)=\partial^*(\alpha^-)=0$. It follows that $\alpha^+\alpha^-$ can be expressed as a combination of  $\alpha ^+_2, \alpha ^-_2\in H^8(M)$. By $\star$ we know that both $\{\alpha^+\alpha^-, \alpha ^+_2\}$ and $\{\alpha^+\alpha^-, \alpha ^-_2\}$ are basis of $H^8(M)/T\cong \Bbb Z^2$, hence  $\{\alpha^+_2, \alpha^-_2\}$ is also a basis of $H^8(M)/T$.  Recall that  $H^{*}(M)/T\otimes \Bbb Q\cong A_4(4)$ or respectively $A_6(4)$, generated by $\Bbb Q[x, y]$ modulo relations $x^4=x^2+y^2=0$ or respectively $x^6=x^2+3y^2=0$. In particular, for any $\beta \neq 0$ of degree $4$, $\beta ^{m-1}\neq 0$, where $m=4 $ or respectively $6$. It is now completely similar to \cite{GH} to derive the multiplicative relations on the generators $\alpha^+_i, \alpha^-_i$ of $H^{*}(M)/T$.

To finish the proof, by the same argument of \cite{GH} on page 456 it follows that, the Stiefel-Whitney class of the normal bundle $\gamma_+$,
$$w_4(\gamma_+)=\alpha^-(\text{mod }2)$$
but the first Pontryagin class $p_1(\gamma_+)=0$ using the Hirzebruch signature theorem, since $\text{sig}(M_\pm)=0$ for homological dimension reasons.

Now unlike \cite{GH}, we consider the composition map of the imbedded fiber $j_-:  \Sph^4\to M$ of the bundle $\pi_-: M \to M_-$ and $\pi_+: M\to M_+$ and the pullback bundle $ (\pi_+\circ j_-)^*(\gamma_+): =\eta $ on the $4$-sphere. Note that $w_4(\eta)= \alpha^-(\text{mod }2)\neq 0$ but the first Pontryagin class $p_1(\eta)=0$. Since $\pi_3(\SO)\cong \Bbb Z$, $p_1(\eta)=0$ implies that $\eta$ is stably trivial and so $w_4(\eta)=0$. A contradiction. This excludes the cases of $A_4(4)\times \Omega \Sph^{17}$ and $A_6(4)\times \Omega \Sph^{25}$, and the desired result follows.

\medskip

\begin{example} Let $\pi: \Sigma \to \Sph^{2m}$ be an $\Sph^{2m-1}$-bundle with nontrivial Euler class, e.g., the unit tangent bundle of the sphere $ \Sph^{2m}$. Note that $\Sigma$ is a rational homology sphere. If $M_\pm \subset \Sph^{2m}$ is a pair of dual submanifolds, then $\pi^{-1}(M_\pm)\subset \Sigma$ is a pair of  dual submanifolds in $\Sigma$ of the same codimensions, but the ratio $\frac {\text{dim } M} {m_++m_-}$ gets doubled.  For any pair $(m_+, m_-)$ where $m_++m_-+1=2m$ and the trivial $m_\pm$-knots $\Sph^{m_\pm}\subset \Sph^{2m}$, the preimages $\pi^{-1}(\Sph^{m_\pm})\subset \Sigma$ is a pair of dual submanifolds with the same codimensions of the knots in $\Sph^{2m}$. Therefore, the ratio  $\frac {\text{dim } M} {m_++m_-}=2$.
\end{example}

\begin{example}
For the Hopf fibration $\pi: \Sph^7\to \Sph^4$, by pullback the isoparametric hypersurfaces and their focal manifolds in $\Sph^4$ we indeed get isoparametric hypersurfaces in $\Sph^7$, e.g, the pullback of Cartan's example $\Sph^3/\Bbb Q_8$ )(with $g=3$) is an isoparametric hypersurface in $\Sph^7$ with $g=6$, clearly diffeomorphic to $\Sph^3\times \Sph^3/\Bbb Q_8$. In a similar vein, for the boundary $\Sph^1\times \Sph^2 \subset \Sph^4$ of the trivial knots, its preimage $\Sph^3 \times \Sph^1\times \Sph^2 \subset \Sph^7$ is an isoparametric hypersurface with $g=4$. The Hopf fibration
$\Sph^{15}\to \Sph^{8}$ also produces isoparametric hypersurface in the total space with $g=4$, but none of $g=6$!
\end{example}

\begin{example} The Milnor spheres $\Sigma$ are $7$-dimensional homotopy spheres which are $\Sph^3$-bundles over $\Sph^4$. Up to orientation reversing, there are exactly $10$ of them are exotic spheres (cf. \cite{GZ}). For the bundle projections $\pi: \Sigma \to \Sph^4$, the pullback
of the dual submanifolds $\Bbb {RP}^2_\pm $,  $\pi^{-1} (\Bbb {RP}^2_\pm )=\Sph^3\times \Bbb {RP}^2_\pm \subset \Sigma $ is a pair of $5$-dimensional submanifolds in the Milnor spheres.
\end{example}

\section{Proof of Theorem 1.3/after Stephan Stolz}

As a natural generalization of isoparametric hypersurface, a Dupin hypersurface $M\subset \Sph^{n+1}$ is a compact hypersurface where the number of distinct principal curvatures are constant and the principal curvatures $\lambda _1(x)\leq  \lambda _2(x)\cdots \leq \lambda _g(x)$ are constant along the leaves of the foliations defined by the eigenspaces of $\lambda_1(x), \cdots, \lambda_g(x)$.  According to M\"unzner and Thorgbersson, $g \in \{1, 2,3, 4, 6\}$. The multiplicities of $\lambda_1, \cdots, \lambda _g$ satisfy that, if $g$ is even, then $m_1=m_3=\cdots =m_{g-1}: =m_+$ and $m_2=m_4=\cdots =m_{g}: =m_-$; and if $g$ is odd, $m_1=m_2=\cdots =m_g$. A longstanding problem was which pair of integers $(m_+,m_-)$ can be realized as the multiplicities of a Dupin hypersurface of $g=4$. Partial results were obtained in \cite{Ab}\cite{GH}\cite{Ta}\cite{Fr}\cite{F3}. The problem was completely solved by Stolz in \cite{St}.

\begin{thm} [Stolz]
Let $M^n\subset \Sph^{n+1}$ be a Dupin hypersurface with $4$ distinct principal curvatures and multiplicities $(m_-, m_+)$. For simplicity let us assume $m_-\leq m_+$. Then $(m_-, m_+)=(2,2)$ or $(4,5)$, or $m_++m_-+1$ is a multiple of $\delta(m_--1)$.
\end{thm}

The proof of Stolz's theorem used heavily stable homotopy theory, which applies identically to the situation of Theorem 1.3 where $\Sigma$ is a homotopy sphere. For reader's convenience we will give a very  brief review on his beautiful proof.

The point of departure  is the Thom-Pontryagin construction which gives a stable map
$$c: \Sigma \to \Sph^1\wedge M$$
since the normal line bundle of $M$ in $\Sigma$ is trivial. In case $\Sigma$ is a homotopy sphere, $c$ gives a stable homotopy class in $\pi_{n}^s(M)$. Note that $c$ is a degree one map, i.e., induces an isomorphism on the homology group $H_{n}(-, \Bbb Z)$. By composition with the bundle projections $p_\pm : M\to M_\pm$ there is a map $(p_+\wedge p_-)\circ c : \Sph^n \to M \to M_+\wedge M_-$. For dimension reason  the image of the map lies in the $n$-skeleton  $(M_+\wedge M_-)^{(n)}$ of the wedge.  Therefore, it gives a stable homotopy class in $\pi_{n}^s((M_+\wedge M_-)^{(n)})$.

An important observation in \cite{St} (cf. Proposition 2.7 therein) is that, if $(m_+, m_-)\neq (5, 4)$, then the wedge $(M_+\wedge M_-)^{(n)}$ desuspends $\ell$ times  for any $\ell \leq m_--1$,  i.e, homotopy equivalent to $\Sph^\ell \wedge X$,  the $\ell$-th suspension of some CW complex $X$. This follows completely elementary depending on the cell structure of $M_\pm, M$ in the homotopy sphere $\Sigma$ forced by the homology groups of $M_\pm$ and $M$.

The geometric construction of May-Milgram-Segal plays a key role (cf. \cite{Mi}). For any connected topological space $X$ and any given integer $k\geq 1$, let $D_{k,2}(X)=(\Sph^{k-1}_+\wedge X\wedge X)/\Bbb Z_2$, where  $\Sph^{k-1}_+$ is the $(k-1)$ sphere equipped with a disjoint base point and $\Bbb Z_2$ acts on   $\Sph^{k-1}$ by the antipodal map and on  $ X\wedge X$ by flipping the factors. One may take $k=\infty$, and let $ D_2(X)$ denote $D_{\infty,2}(X)$ for simplicity. It is well-known that, $ D_2(\Sph^\ell )=\Sph^\ell\wedge \Bbb {RP}^\infty_\ell$, the $\ell$-times suspension of the stunted real projective space $ \Bbb {RP}^\infty_\ell= \Bbb {RP}^\infty/\Bbb {RP}^{\ell-1}$,  by collapsing the subspace $\Bbb {RP}^{\ell-1}\subset \Bbb {RP}^{\infty}$. By \cite{Ma}\cite{Mi},  for any $q< 3r-1$, there is a generalized EHP exact sequence
$$\pi_q(X)\to \pi_q^s(X) \to \pi_q(D_2(X))\to \pi_{q-1}(X)\to \pi_{q-1}^s(X)$$
whenever $X$ is $(r-1)$-connected and $r\geq 2$. The homomorphism  from $\pi_q^s(X)$ to $\pi_q(D_2(X))$  in the sequence is called the Hopf invariant $H$.

Now let $X$ be the CW complex such that $(M_+\wedge M_-)^{(n)}=\Sph^\ell \wedge X$ for any $\ell < m_-$.  Let $m=m_++m_-$. Note that $n=2m$. By the previous section we know that $m$ is odd. Moreover, by the
homology groups calculation of $M_\pm$ it follows that, $M_+\wedge M_-$ is $(m-1)$-connected. Moreover, $H_m(M_+\wedge M_-; \Bbb Z)\cong \Bbb Z$. Therefore,
$\pi_m(M_+\wedge M_-)\cong \Bbb Z$ by the Hurewicz theorem. Let $i: \Sph^m\subset (M_+\wedge M_-)^{(2m)}$ be the inclusion of the bottom dimensional cell.   Let $f: \Sph^{2m-\ell}\to X$ be the stable homotopy class such that $\Sph^\ell \wedge f=(p_+\wedge p_-)\circ c: \Sph^{2m}\to  (M_+\wedge M_-)^{(2m)}=\Sph^\ell \wedge X$.

 An important step in \cite{St}, based on the framed bordism theory and Brown-Kervaire invariant, is to prove that the Hopf invariant $H(f)$ is not zero. Applying the $D_2$ functor to $i: \Sph^{m-\ell} \subset X$ we get a map
$$D_2(i): \Sph^{m-\ell} \wedge \Bbb {RP}^\infty _{m-\ell}=D_2( \Sph^{m-\ell})\to D_2( X)$$
From the cell structure of $X$ it is easy to see that $D_2(i)$ induces an epimorphism on the stable homotopy groups $\pi_{2m-\ell}(-)$.

On the other hand, there is a diagonal map $\Sph^\ell \wedge D_2(X)\to D_2(\Sph^\ell \wedge  X)$ which induces a homomorphism $\Delta_* : \pi_{2m} (\Sph^\ell \wedge D_2(X))\to \pi_{2m}(D_2(\Sph^\ell \wedge X))$. It is a technical result that $\Delta_* $ commutes with Hopf invariants under the suspension isomorphism (cf. Lemma 4.8 in \cite{St}) on the stable homotopy groups, therefore, there is an element $\alpha \in  \pi_{2m} (\Sph^\ell \wedge D_2(X))$ such that $\Delta_*(\alpha) =H(f) \neq 0$, the Hopf invariant of $\Sph^\ell \wedge f$.

Since $D_2(i)$ induces an epimorphism on the stable homotopy groups (here we are often shifting the degree in the category of stable homotopy), we may write $\alpha =(\Sph^\ell\wedge D_2(i))_*(s)$ for some $s\in \pi_{2m}(\Sph^\ell\wedge D_2(\Sph^{m-\ell})=\pi_{m}^s( \Bbb {RP}^\infty _{m-\ell})$. Note that $D_2(\Sph^\ell \wedge X)$ is $(2m-1)$-connected. It follows that $H(f)$ induces a nonzero homomorphism on the $2m$-th homology groups by the Hurewicz theorem. Therefore,
the stable map $s: \Sph^ m \to  \Bbb {RP}^\infty _{m-\ell}$ also induces a nonzero homomorphism on the $m$-th integral homology groups, note here there is a degree shifting due to the suspension isomorphism. From homotopy theory it is not difficult to see that $s$ is homotopy to a map into the $m$-skeleton $\Bbb {RP}^m _{m-\ell}$  such that, its composition with the collapsing map onto the sphere $\Sph^m=\Bbb {RP}^m _{m-\ell}/\Bbb {RP}^m _{m-1}$
$$\Sph^m \to  \Bbb {RP}^m _{m-\ell} \to \Sph^m$$
has degree one (cf. page 264 in \cite{St}). In other words, the top cell of $\Bbb {RP}^m _{m-\ell} $ splits off. It is now a classical result of Adams \cite{Ad} in solving the vector fields problem on the spheres, that $m+1$ is divisible by $\delta(\ell)$.

\bibliographystyle{amsalpha}

\begin{thebibliography}{99}

\bibitem{Ab} Abresch, U., {\it Isoparametric hypersurfaces with four or six
distinct principal curvatures}, Math. \ Ann., {\bf 264 }(1983), 283-302

\bibitem{Ad} Adams, J. F., {\it Vector fields on spheres}, Ann. of Math. {\bf 75} (1962), 603-632





\bibitem{Ca1}  Cartan, \'E., {\it Sur des familles remarquables d'hypersurfaces
isoparametriques dans les espaces spheriques}, Math. \ Z., {\bf 45}(1939),
335-367

\bibitem{Ca2} Cartan, \'E., {\it Familles de surfaces isoparam\'etriques dans les espaces \`a courbure constante}, Annali di Mat. 17
(1938), 177-191 (see also in Oeuvres Compl\'etes, Partie III, Vol. 2, 1431-1445).

\bibitem{Ca3}  Cartan, \'E., {\it Sur quelque familles remarquables d'hypersurfaces}, in C.R. Congr\'es Math. Li\'ege, 1939, 30-41 (see also in Oeuvres Compl\'etes, Partie III, Vol. 2, 1481-1492).

\bibitem{Ca4} Cartan, \'E., {\it Sur des familles d¡¯hypersurfaces isoparam\'etriques des espaces sph\'eriques \`a 5 et \`a 9 dimensions}, Revista Univ. Tucuman, Serie A, 1 (1940), 5-22 (see also in Oeuvres Compl\'etes, Partie III, Vol. 2, 1513-1530).

\bibitem{Chern} Chern S.-S., {\it An introduction to Dupin submanifolds}, in Differential Geometry: A Symposium in Honour of Manfredo do Carmo (Rio de Janeiro, 1988), Editors H.B. Lawson and K. Tenenblat, Pitman Monographs Surveys Pure Appl. Math., Vol. 52, Longman Sci. Tech., Harlow, 1991, 95-102.

\bibitem{CC}    Cecil T., Chern S.-S., {\it Dupin submanifolds in Lie sphere geometry}, in Differential Geometry and Topology (Tianjin, 1986-87), Editors B. Jiang et al., Lecture Notes in Math., Vol. 1369, Springer, Berlin ¨C New York, 1989, 1-48.

\bibitem{CCJ1} Cecil T., Chi Q.-S., Jensen G., {\it Isoparametric hypersurfaces with four principal curvatures}, Ann. of Math. (2)
166 (2007), 1-76.

\bibitem{CCJ2}  Cecil T., Chi Q.-S., Jensen G., {\it Dupin hypersurfaces with four principal curvatures. II}, Geom. Dedicata 128
(2007), 55-95.


\bibitem{DN}  Dorfmeister, J. and Naher, E., Isoparametric hypersurfaces,
case $g=6, m=1$, {\it Comm. \ Algebra}, {\bf 13}(1985), 2299-2368

\bibitem{Fr} Fang, F., {\it Multiplicities of principal curvatures of isoparametric
hypersurfaces}, preprint, Max-Planck Institute series 96-80


\bibitem{F2} Fang, F., {\it Topology of Dupin hypersurfaces with six
distinct principal curvatures},   Math. Z., {\bf 231} (1999), 533-555.


\bibitem{F3}   Fang, F., {\it On the topology of isoparametric hypersurfaces
with four distinct principal curvatures}, Proc. Amer. Math.Soc., {\bf 127}
(1999), 259-264



\bibitem{FKM} Ferus, D., Karcher, H. and M\"unzner, H., {\it Cliffordalgebren
und neue isoparametrische hyperfl\"achen}, Math.\ Z., {\bf 177}(1981),
479-502

\bibitem{GH} Grove, K., Halperin, S., {\it  Dupin hypersurfaces, group actions and the double mapping cylinder,}  J. Differential Geom.  {\bf 26}  (1987), 429-459.

\bibitem{GZ}  Grove, K., Ziller, W., {\it Curvature and symmetry of Milnor spheres}, Ann. of Math., {\bf   152}  (2000), 331-367.



\bibitem{HL} Hsiang W.-Y., Lawson H. B., Jr., {\it Minimal submanifolds of low cohomogeneity}, J. Differential Geom. 5
(1971), 1-38.

\bibitem{Ma} May, J.P., {\it The geometry of iterated loop spaces,}  Lect. Notes Math. 271, Springer Verlag, 1972

\bibitem{Mi} Milgram,  R. J., {\it Unstable homotopy from the stable point of view}, Lect. Notes Math. {\bf 368}, Springer Verlag, 1974.

\bibitem{Miy} Miyaoka, J. {\it
Isoparametric hypersurfaces with (g,m)=(6,2)}, Ann. of Math. {\bf 177}  (2013),    53-110.

\bibitem{Mu} M\"unzner, H.F., {\it Isoparametric hyperfl\"achen in sph\"aren
I, II, } Math. \ Ann., {\bf 251}(1980), 57-71; {\bf 256}(1981), 215-232


\bibitem{St}  Stolz, S., {\it Multiplicities of Dupin hypersurfaces}, Invent. Math. {\bf 138} (1999), 253-279.

\bibitem{Ta} Tang, Z., {\it Isoparametric hypersurfaces with four distinct
principal curvatures}, Chinese Science Bulletin., {\bf 36}(1991),
1237-1240

\bibitem{Th} Thorbergsson, G., {\it Dupin hypersurfaces},  Bull. London Math. Soc., {\bf  15 } (1983), 493-498.

\bibitem{Wa}  Wang, Q., {\it On the topology of Clifford isoparametric
hypersurfaces,} J. Differential Geometry., {\bf 27}(1988), 55-66


\bibitem{Ya}
 Yau, S.T., {\it Chern, a great geometer  of the twentieth century},
International press 1992


\end{thebibliography}

\end{document}